\documentclass{article}
\usepackage{amsmath}
\usepackage{amsfonts}
\usepackage{amssymb}
\DeclareMathOperator{\lcm}{lcm}

\usepackage[dvipsnames,svgnames,table]{xcolor}
\usepackage[colorlinks=true,linkcolor=blue,citecolor=ForestGreen]{hyperref}
\usepackage{orcidlink}
\usepackage{array}
\usepackage{xspace}

\usepackage{amsthm}
\usepackage[capitalise]{cleveref}

\newtheorem{theorem}{Theorem}[section]

\newtheorem{proposition}[theorem]{Proposition}

\newtheorem{conjecture}[theorem]{Conjecture}
\newtheorem{lemma}[theorem]{Lemma}
\theoremstyle{definition}

\theoremstyle{remark}

\title{A Question of Erd\H{o}s and Graham on Covering Systems}
\author{Sarosh Adenwalla\thanks{Department of Computer Science, University of Liverpool, UK, \texttt{sarosh.adenwalla@liverpool.ac.uk}, \orcidlink{0009-0009-8582-1281}}}
\date{October 2024}

\begin{document}

\maketitle
\begin{abstract}
    Erd\H{o}s and Graham \cite{erdHos1980old} asked if there exists an $n$ such that the divisors of $n$ greater than 1 are the moduli of a distinct covering system with the following property: If there exists an integer which satisfies two congruences in the system, $a\mod d$ and $a'\mod d'$, then $\gcd(d,d')=1$. We show that such an $n$ does not exist. This problem is part of Problem \# 204 on the website www.erdosproblems.com, compiled and maintained by Thomas Bloom.
\end{abstract}

\section{Introduction}
A \textit{covering system} is a set of congruences, $\{a_1\mod d_1,\ldots,a_t\mod d_t\}$, such that every integer is equivalent to some $a_i\mod d_i$ for some $1\leq i\leq t$. A covering system is \textit{distinct} if $1<d_1<d_2<\ldots<d_t$. Two congruences \textit{overlap} if there is an integer equivalent to both congruences.

There have been many questions asked about distinct covering systems, the most well-known being whether $d_1$ can be arbitrarily large (asked by Erd\H{o}s \cite{erdos1950integers}), whether $d_i|d_j$ must occur for some $i\neq j$ (asked by Schinzel \cite{schinzel1967reducibility}) and whether it is possible for all $d_i$ to be odd (asked by Erd\H{o}s and Selfridge \cite[Section F13]{guy2004unsolved}). Hough answered the first question in \cite{hough2015solution}, proving that $d_1$ must be at most $10^{16}$. This bound was improved to $616$,$000$ in \cite{balister2022erdHos} and the second question was answered affirmatively in \cite{balister2022erdHos} as well. The third question is still open, however there has been much progress towards an answer. It is known that for any odd distinct covering system, there must be some $d_i$ divisible by 3 \cite{Hough} and in fact it has been shown that there must be some $d_i$ divisible by 9, or some $d_i$ divisible by 3 and some $d_j$ divisible by 5 \cite{balister2022erdHos}. It is also known that not all $d_i$ can be square-free \cite{Morris} in such a covering system.

We note that using the result in \cite{Hough}, that any distinct covering system must have a modulus divisible by 2 or 3, does shorten our proof somewhat. However, as the result does not require such a powerful tool and in the interests of having a self-contained proof, we avoid using it.

We say a set of congruences (or a congruence set), $\{a_1\mod d_1,\ldots,a_t\mod d_t\}$,  is \textit{coprime disjoint} (CD) if whenever two congruences in the set $a_i\mod d_i$ and $a_j\mod d_j$ overlap for $i\neq j$, we have $\gcd(d_i,d_j)=1$. All congruence sets that we consider will be distinct.

Furthermore, we say that $n$ is \textit{non-intersecting} if there exist integers $\{a_d \,:\, d|n,d>1\}$ such that $\{a_d\mod d\,:\, d|n, d>1\}$ is a CD congruence set. We say this is a CD congruence set of $n$.

Finally, we say $n$ is \textit{CD covering} if there exist integers $\{a_d \,:\, d|n,d>1\}$ such that $\{a_d\mod d\,:\, d|n, d>1\}$ is a CD distinct covering system. It is clear that any CD covering $n$ is non-intersecting.

Erd\H{o}s and Graham \cite{erdHos1980old} posited that a CD covering $n$ probably does not exist and determining this is part of Problem \#204 in \cite{Bloom}. We prove that this is correct, that is, there are no CD covering $n$.

We also provide a necessary condition for $n$ to be non-intersecting and prove certain families of $n$ are non-intersecting.

For convenience, from here on when referring to the divisors of $n$ we will not include 1. 

\textbf{Acknowledgements:} We wish to thank Stijn Cambie for his insights and examples on a related question. We additionally wish to thank the referee for their helpful comments and corrections.

\section{Preliminaries}

 Let $A=\{a_1\mod d_1,\ldots,a_t \mod d_t\}$ be a congruence set. An integer, $x$, is \textit{covered} by $a_i\mod d_i$ if $x\equiv a_i\mod d_i$. Similarly, $x$ is \textit{covered} by $A$ if there exists $1\leq i\leq t$ such that $x\equiv a_i\mod d_i$. A residue class, $a \mod l$, is covered by $A$ if every integer $x\equiv a\mod l$ is covered by $A$. Note that if $\lcm(d_1,\dots,d_t)|l$ and some integer $x\equiv a \mod l$ is covered by $A$, then the residue class $a\mod l$ is covered by $A$. This follows as $$a+lr=x=a_i+d_ir'$$ for some $1\leq i\leq t$ and $r,r'\in\mathbb{Z}$. As $d_i|l$, we have $$a+lk=a_i+d_i(r'+\frac{l}{d_i}(k-r))$$ for any $k\in\mathbb{Z}$.

 Let $A(n)$ be the number of integers that are less than or equal to $n$ and covered by $A$. Then the \textit{density} of integers that are covered by $A$ is $$\delta(A)=\lim_{n\rightarrow\infty} \frac{A(n)}{n}.$$ 
 
 Let $D=\lcm(d_1,\ldots,d_t)$ and $m$ be the number of residue classes mod $D$ that are covered by $A$. Then as $A(n)=m\frac{n}{D}+O(D)$, we see that $$\delta(A)=\frac{m}{D}.$$ So the fraction of residues classes mod $D$ that are covered by $A$ is equal to $\delta(A)$. Clearly $\delta(A)=1$ if and only if $A$ is a covering set.

We adapt the proof of \cite[Lemma 2.3]{simpson1986exact} in order to obtain an expression for the density of a CD congruence set. 

\begin{lemma}\label{good}
    Let $A=\{a_1\mod d_1,\ldots,a_t \mod d_t\}$ be a CD congruence set. Then 
    \begin{equation}\label{1}
        \delta(A)= \sum_{i=1}^t \frac{1}{d_i}-\sum_{\substack {1\leq i_1<i_2\leq t,\\ \gcd(d_{i_1},d_{i_2})=1}}\frac{1}{d_{i_1}\cdot d_{i_2}}+\sum_{\substack {1\leq i_1<i_2<i_3\leq t,\\ \prod_{1\leq j<k\leq 3}\gcd(d_{i_j},d_{i_k})=1}}\frac{1}{d_{i_1}\cdot d_{i_2}\cdot d_{i_3}}-\ldots,
         \end{equation}
    where the $s$-th sum is over every set of $s$ pairwise coprime $d_i$.
\end{lemma}

Note that we consider empty sums to be 0. In particular, if the moduli of $A$ are the divisors of a non-intersecting $n$, then the number of sums in \eqref{1} is equal to the number of distinct prime divisors of $n$. 

\begin{proof}[Proof of Lemma \ref{good}]
    Let $A_i$ be the set of residue classes mod $D$ covered by $a_i\mod d_i$. Let $N(i_1,\ldots,i_s)=|\bigcap_{j=1}^s A_{i_j}|$. By the inclusion-exclusion principle, the number of residue classes covered by $A$ is $$\sum_{i=1}^t N(i)-\sum_{1\leq i<j\leq t}N(i,j)+\sum_{1\leq i<j<k\leq t} N(i,j,k)+\ldots.$$

    Let $e(i,j)=1$ if $a_i\mod d_i$ and $a_j\mod d_j$ overlap and $0$ otherwise. As $A$ is a CD congruence set, $a_i\mod d_i$ and $a_j\mod d_j$ cannot overlap if $\gcd(d_i,d_j)>1$. If $\gcd(d_i,d_j)=1$, then by the Chinese Remainder Theorem, $a_i\mod d_i$ and $a_j\mod d_j$ must overlap. Therefore $e(i,j)=1$ if and only if $\gcd(d_i,d_j)=1$.
    
    By the Chinese Remainder Theorem, if $\gcd(d_{i_j},d_{i_k})=1$ for $1\leq j<k\leq s$ then there is a unique residue class mod $d_{i_1}\cdot\ldots\cdot d_{i_s}$ that satisfies all congruences $a_{i_1}\mod d_{i_1},\ldots, a_{i_s}\mod d_{i_s}$. It follows that there are exactly $\frac{D}{d_{i_1}\cdot\ldots\cdot d_{i_s}}$ residue classes mod $D$ in $\bigcap_{j=1}^s A_{i_j}$.
    
    Therefore $\bigcap_{j=1}^s A_{i_j}\not= \emptyset$ if and only if $\prod_{1\leq j<k\leq s} e(i_j,i_k)=1$. Then the number of residue classes mod $D$ that is covered by $A$ is 
    $$ \sum_{i=1}^t \frac{D}{d_i}-\sum_{\substack {1\leq i_1<i_2\leq t,\\ \gcd(d_{i_1},d_{i_2})=1}}\frac{D}{d_{i_1}\cdot d_{i_2}}+\sum_{\substack {1\leq i_1<i_2<i_3\leq t,\\ \prod_{1\leq j<k\leq 3}\gcd(d_{i_j},d_{i_k})=1}}\frac{D}{d_{i_1}\cdot d_{i_2}\cdot d_{i_3}}-\ldots,$$
    and dividing by $D$ gives \eqref{1}.
\end{proof}

Note that Lemma \ref{good} shows that the density of integers covered by a CD congruence set is entirely determined by the moduli and not the specific residues associated with them.

It will often be convenient to have a simpler bound. The following result is often stated without proof, we include one for clarity.

\begin{lemma}\label{P1}
    For any covering set, $A=\{a_1\mod d_1,\ldots,a_t\mod d_t\}$, we have $$\sum_{i=1}^t \frac{1}{d_i}\geq 1.$$
\end{lemma}

\begin{proof}
    The residue classes mod $D$ covered by $a_i\mod d_i$ are exactly $a_i + k\cdot d_i\mod D$ for $1\leq k\leq \frac{D}{d_i}$. Therefore there are $\frac{D}{d_i}$ residue classes mod $D$ covered by $a_i\mod d_i$. So there are at most $ \sum_{i=1}^t \frac{D}{d_i}$ residue classes mod $D$ covered by $A$. It follows that $$\delta(A)\leq \sum_{i=1}^t \frac{1}{d_i}.$$ As $A$ is a covering set, we have $\delta(A)=1$ and so the result follows.
\end{proof}

Let $n=p_1^{\alpha_1}p_2^{\alpha_2}\ldots p_s^{\alpha_s}$. We will repeatedly make use of the following identity.

\begin{proposition}[Folklore]
    \begin{equation}\label{EulerProd}\sum_{d|n,d>1}\frac{1}{d}=-1+\sum_{d|n}\frac{1}{d}=-1+\prod_{i=1}^s\sum_{j=0}^{\alpha_i}\frac{1}{p_i^j},\end{equation}
\end{proposition}
\begin{proof}
    The first equality of \eqref{EulerProd} is clear to see, so it suffices to prove the second equality. Note that $$\prod_{i=1}^s\sum_{j=0}^{\alpha_i}\frac{1}{p_i^j}=\prod_{i=1}^s\left(1+\frac{1}{p_i}+\frac{1}{p_i^2}+\ldots+\frac{1}{p_i^{\alpha_i}}\right).$$ 
    After expanding the bracket on the RHS the terms are all the reciprocals of integers of the form $p_1^{\beta_1}p_2^{\beta_2}\ldots p_s^{\beta_s}$ for $0\leq \beta_i\leq \alpha_i$. Additionally, any two of these terms are distinct, as for any two of them there must exist a $1\leq i\leq s$ such that different powers of $p_i$ divide their denominators. Therefore, letting $S=\{p_1^{\beta_1}p_2^{\beta_2}\ldots p_s^{\beta_s}\,:\,\ 0\leq \beta_i\leq \alpha_i\}$, we see that $$\prod_{i=1}^s\sum_{j=0}^{\alpha_i}\frac{1}{p_i^j}=\sum_{k\in S} \frac{1}{k}.$$  As the divisors of $n$ are precisely the integers of the form $p_1^{\beta_1}p_2^{\beta_2}\ldots p_s^{\beta_s}$ for $0\leq \beta_i\leq \alpha_i$, we have $$\sum_{k\in S}\frac{1}{k}=\sum_{d|n}\frac{1}{d},$$ and this completes the proof. 
\end{proof}

We also often use the formula for the sum of an infinite geometric series to bound the RHS of \eqref{EulerProd}.

\section{CD covering \texorpdfstring{$n$}{n}}

\begin{lemma}\label{nice}
    Let $n>1$ be non-intersecting. If $p$ is the smallest prime that divides $n$, then $\frac{n}{p}$ has less than $p$ distinct prime divisors.
    
\end{lemma}

\begin{proof}
    Let $n$ be non-intersecting and $\frac{n}{p}$ have at least $p$ distinct prime divisors. So $q_1,\ldots q_p$ are distinct primes that divide $\frac{n}{p}$ and therefore $pq_1,\ldots pq_p$ all divide $n$. 

    We claim that $a_{pq_i}\equiv a_{pq_j}\mod p$ for some $i\neq j$. If not, then all $a_{pq_i}\mod p$ are distinct for $1\leq i\leq p$. As there are only $p$ distinct residues mod $p$, by the pigeonhole principle $a_p\equiv a_{pq_i}\mod p$ for some $1\leq i\leq p$. However then $a_{pq_i}\equiv a_{p}\mod p$ and $a_{pq_i}\equiv a_{pq_i}\mod pq_i$, which contradicts $n$ being non-intersecting as $\gcd(pq_i,p)=p>1$.

    So $a_{pq_i}\equiv a_{pq_j}\equiv b\mod p$ for some $i\neq j$ and $b\in \{0,1,\ldots,p-1\}$. By the Chinese Remainder Theorem, there exists $c$ such that $c\equiv \frac{a_{pq_i}-b}{p}\mod q_i$ and $c\equiv \frac{a_{pq_j}-b}{p}\mod q_j$ as $\gcd(q_i,q_j)=1$. So $pc+b\equiv a_{pq_i}\mod pq_i$ and $pc+b\equiv a_{pq_j}\mod pq_j$. This contradicts $n$ being non-intersecting as $\gcd(pq_i,pq_j)=p>1$.
\end{proof}

By Lemma \ref{nice}, we have two possible cases for non-intersecting $n$, based on whether $p\nmid \frac{n}{p}$ or $p|\frac{n}{p}$.
   \begin{enumerate}
        \item\label{item:Case1} $n=pq_1^{\alpha_1}\ldots q_{s}^{\alpha_{s}}$ for $0\leq s\leq p-1$ where $p<q_1<\ldots<q_s$ are distinct primes and $\alpha_i\geq 1$ for $1\leq i\leq s$.

        or
        \item\label{item:Case2} $n=p^{\alpha_p}q_1^{\alpha_1}\ldots q_{s}^{\alpha_{s}}$ for $0\leq s\leq p-2$ where $p<q_1<\ldots<q_s$ are distinct primes, $\alpha_p\geq 2$ and $\alpha_i\geq 1$ for $1\leq i\leq s$.
    \end{enumerate} 

We can now prove that there are no CD covering $n$.

\begin{theorem}\label{T1}
    There does not exist any $n$ such that a congruence set of the form $\{a_d\mod d\,:\,d|n,d>1\}$, where $a_d\mod d$ and $a_{d'}\mod d'$ overlap only if $\gcd(d,d')=1$, is a covering set.
\end{theorem}
\begin{proof}
Let $p$ be the smallest prime that divides $n$. As noted above, there are two cases for a non-intersecting $n$, depending on whether $p^2|n$. 

\smallskip

\noindent\textbf{Case \ref{item:Case1} ($n=pq_1^{\alpha_1}\ldots q_{s}^{\alpha_{s}}):$}
We split into two cases based on the parity of $n$.
\begin{itemize}
    \item $\textbf{Case 1a ($n$ odd)}:$

Observe that $\sum_{d|n,d>1}\frac{1}{d}=-1+\sum_{d|n}\frac{1}{d}$. Then by \eqref{EulerProd}, 
$$\sum_{d|n,d>1}\frac{1}{d}=-1+\left(1+\frac{1}{p}\right)\prod_{i=1}^s \sum_{j=0}^{\alpha_i}\frac{1}{q_i^j}< -1+\frac{p+1}{p}\prod_{i=1}^s \frac{q_i}{q_i-1}.$$ Note that $\frac{x}{x-1}$ is a decreasing function for $x>1$. As $n$ is odd, we have $q_i>p\geq 3$ and so $q_i>p+1$ for $1\leq i\leq s$. We can then bound $\frac{q_i}{q_i-1}$ above by $\frac{p+1+i}{p+i}$ and as $s\leq p-1$,
$$ \sum_{d|n,d>1}\frac{1}{d}<-1+\frac{p+1}{p}\cdot\frac{p+2}{p+1}\cdot\frac{p+3}{p+2}\cdot\ldots\cdot\frac{p+(p-1)+1}{p+(p-1)}= -1+2=1. $$

Therefore by Lemma \ref{P1}, $n$ is not a covering set.

 \item $\textbf{Case 1b ($n$ even)}:$

    So $p=2$ and either $n=2$ or $n=2q^{\alpha}$ for a prime $q\geq 3$. If $n=2$ then clearly $$\sum_{d|n,d>1}\frac{1}{d}=\frac{1}{2}<1,$$ and so by Lemma \ref{P1}, $n$ is not a covering set.
    
    If $n=2q^{\alpha}$ then by Lemma \ref{good}, the density of integers covered by a CD congruence set, $A$, with moduli that are the divisors of $n$ is 
    $$\delta(A)=\left(\frac{1}{2}+\sum_{i=1}^\alpha \frac{1}{q^i}+\sum_{i=1}^\alpha \frac{1}{2q^i}\right)-\sum_{i=1}^\alpha \frac{1}{2q^i}=\frac{1}{2}+\sum_{i=1}^\alpha \frac{1}{q^i},$$
     as the only pairs of divisors, $(d_1,d_2)$, of $n$ such that $\gcd(d_1,d_2)=1$ is $(2,q^i)$ for $1\leq i\leq \alpha$. Then,
    $$\delta(A)<\frac{1}{2}+\sum_{i=1}^\infty \frac{1}{q^i}=\frac{1}{2}+\frac{1}{q-1}\leq \frac{1}{2}+\frac{1}{3-1}=1.$$
    As $\delta(A)<1$, we see that $n$ is not a covering set.
\end{itemize}

\smallskip

\noindent    \textbf{Case \ref{item:Case2} ($n=p^{\alpha_p}q_1^{\alpha_1}\ldots q_{s}^{\alpha_{s}}$):}

As before we use \eqref{EulerProd},
$$
\sum_{d|n,d>1}\frac{1}{d}=-1+\left(\sum_{j=0}^{\alpha_p}\frac{1}{p^j}\right)\prod_{i=1}^s\sum_{j=0}^{\alpha_i}\frac{1}{q^j} <-1+\left(\sum_{j=0}^{\infty}\frac{1}{p^j}\right)\prod_{i=1}^s\sum_{j=0}^{\infty}\frac{1}{q^j}.$$ As $\frac{x}{x-1}$ is a decreasing function and $q_i>p$ for $1\leq i\leq s$, we can bound $\frac{q_i}{q_i-1}$ above by $\frac{p+i}{p+i-1}$. Therefore $s\leq p-2$ implies $$\sum_{d|n,d>1}\frac{1}{d}<-1+\frac{p}{p-1}\prod_{i=0}^s \frac{q_i}{q_i-1}\leq -1+\frac{p}{p-1}\cdot \frac{p+1}{p}\cdot\ldots\cdot\frac{p+(p-2)}{p+(p-2)-1}=1.$$
So by Lemma \ref{P1}, $n$ is not a covering set.
\end{proof}

\section{Non-intersecting \texorpdfstring{$n$}{n}}

It was also asked in \cite{erdHos1980old}, for a given $n$ and any CD congruence set of $n$, what is the minimum density of integers that do not satisfy any of the congruences in the set. This is equivalent to asking what is the maximum density of integers that satisfy a congruence in the set. Lemma \ref{good} gives this density for any non-intersecting $n$. 

It remains to confirm which $n$ have divisors that are the moduli of a CD congruence set. We answer this question for some families of integers. 

Lemma \ref{nice} goes part way to answering this, giving a necessary condition for $n$ to be non-intersecting. We show that if $n=p^k$ or $n=qp^k$ for $k\geq 1$ and primes $p,q$, then $n$ is non-intersecting.

\begin{proposition}
    Let $n=p^k$ for a prime $p$ and $k\geq 1$. Then $n$ is non-intersecting, and the density of any CD congruence set of $n$ is $\frac{p^k-1}{p^k(p-1)}$.
\end{proposition}
\begin{proof}
The set $\{ p^{i-1}\mod p^i \,:\, 1\leq i\leq k\}$ is CD as no integer is equivalent to two congruences from the set. To see this, let $x\equiv p^{i-1}\mod p^i$ and $x\equiv p^{j-1}\mod p^j$ for $i\neq j$. The first congruence implies that $p^{i-1}$ is the highest power of $p$ that divides $x$ and similarly, the second congruence implies that $p^{j-1}$ is the highest power of $p$ that divides $x$. This means that $i=j$, which is a contradiction.

By putting the divisors of $n$ into \eqref{1}, we see that the result equals $$\sum_{i=1}^{k}\frac{1}{p^i}=\frac{p^k-1}{p^k(p-1)},$$ and it follows from Lemma \ref{good} that this is the density of integers that satisfy a congruence in the set.  
\end{proof}

\begin{proposition}
    Let $n=qp^k$ for distinct primes $p,q$ and $k\geq 1$. Then $n$ is non-intersecting if and only if $k=1$ or $p>2$, and the density of any CD congruence set of $n$ is $\frac{p^k-1}{p^k(p-1)}+\frac{1}{q}$.
\end{proposition}

\begin{proof}
Let $n=qp^k$ for distinct primes $p,q$ and $k\geq 1$. When $k\geq 2$ and $p=2$, we see that $\frac{n}{2}=2^{k-1}q$ has $2$ distinct prime divisors, so by Lemma \ref{nice} it follows that $n$ is not non-intersecting.

Now consider the case where $k=1$ or $p>2$. When $k=1$, we see that $n$ is the product of two distinct primes, $p$ and $q$, and without loss of generality we assume $p$ is the larger prime. Then $n=pq$ and $p>q\geq 2$. So in either case, we can assume that $p>2$.

Let $p>2$ and $a$ be an integer such that $a\not \equiv 0,q^{-1}\mod p$ (this is possible as $p\geq 3$). Then let the congruence set be $$0\mod q, \{p^{i-1}+1\mod p^i\,:\, 1\leq i\leq k\}, \{aqp^{j-1}+1\mod qp^j\,:\, 1\leq j\leq k\}.$$ We claim that this is a CD congruence set. Clearly if $x\equiv aqp^{j-1}+1\mod qp^j$ and $x\equiv 0\mod q$ then $q|1$ which is a contradiction. It remains to check whether the two sets of congruences overlap with each other. 

First, we assume that $p^{i-1}+1\mod p^i$ and $p^{j-1}+1\mod p^j$ overlap for $i\neq j$. Therefore $$p^{i-1}+rp^i=p^{j-1}+tp^j$$ for $r,t\in \mathbb{Z}$. Then the highest power of $p$ that divides the LHS is $p^{i-1}$ and the highest power of $p$ that divides the RHS is $p^{j-1}$, which is a contradiction.

Next, we assume that $aqp^{i-1}+1\mod qp^i$ and $aqp^{j-1}+1\mod qp^j$ overlap for $i\neq j$. Therefore $$aqp^{i-1}+rqp^i=aqp^{j-1}+tqp^j$$ for $r,t\in \mathbb{Z}$. Then, as $a\not \equiv 0\mod p$, the highest power of $p$ that divides the LHS is $p^{i-1}$ and the highest power of $p$ that divides the RHS is $p^{j-1}$. This is a contradiction.

Finally, we assume that $p^{i-1}+1\mod p^i$ and $aqp^{j-1}+1\mod qp^j$ overlap. Then $$p^{i-1}+rp^i=aqp^{j-1}+tqp^j$$ for $r,t\in \mathbb{Z}$. So $p^{i-1}(1+rp)=qp^{j-1}(a+tp)$. Then the highest power of $p$ that divides the LHS is $p^{i-1}$ and the highest power of $p$ that divides the RHS is $p^{j-1}$. It follows that $i=j$ and so $(1+rp)=q(a+tp)$. However, this implies that $1\equiv qa\mod p$, which contradicts $a\not\equiv q^{-1}\mod p$. 

So the divisors of $qp^k$ are the moduli of a CD congruence set. By putting the divisors of $n$ into $\eqref{1}$, we see that the result equals \begin{align*}-1+\left(1+\frac{1}{q}\right)\sum_{i=0}^k \frac{1}{p^i}-\sum_{i=1}^k \frac{1}{qp^i}&=-1+\frac{q+1}{q}+\left(1+\frac{1}{q}\right)\sum_{i=1}^k \frac{1}{p^i}-\sum_{i=1}^k \frac{1}{qp^i}\\ &=\frac{1}{q}+\frac{p^k-1}{p^k(p-1)}, \end{align*} and it follows from Lemma \ref{good} that this is the density of integers that satisfy a congruence in the set.  
\end{proof}

By Lemma \ref{nice}, this covers all even $n$ whose divisors are the moduli of a CD congruence set.

\section{Conclusion}

The next step would be to provide a necessary and sufficient characterisation for $n$ to be non-intersecting.

By Theorem \ref{T1}, if the divisors of $n$ are used in \eqref{1} and the answer is greater than or equal to 1, then $n$ is not non-intersecting. However, this is not an if and only if condition as can be seen by considering the divisors of $n=20$. Using $\{2,4,5,10,20\}=\{d_1,d_2,d_3,d_4,d_5\}$ in \eqref{1} gives a result of $\frac{19}{20}$. At the same time, Lemma \ref{nice} shows that as $\frac{20}{2}=10$ has more than 1 prime factor, $n$ cannot be non-intersecting. 

We conjecture that Lemma \ref{nice} is a sufficient condition for $n$ to be non-intersecting, as well as a necessary one. 

\begin{conjecture}
    Let $p$ be the smallest prime that divides $n$. If $\frac{n}{p}$ has less than $p$ distinct prime divisors, then $n$ is non-intersecting. 
\end{conjecture}

A proof of this conjecture has been claimed in \cite{jia2025resolving}.

A further avenue is to characterise when $\{d_1,\ldots,d_t\}$ are the moduli of a CD congruence set. A similar argument to Lemma \ref{nice} implies that if $\{d_1,\ldots,d_t\}$ are the moduli of a CD congruence set then, for any integer $m$, at most $m$ moduli form a set where the terms have pairwise gcd equal to $m$. Stijn Cambie has constructed examples that show this is not a sufficient condition for $\{d_1,\ldots,d_t\}$ to be non-intersecting, e.g. $\{3,6,12,18,30,42\}$. 

A more general question is asked in \cite{erdHos1980old}: what is the maximum density of integers covered by a congruence set with a given set of moduli $\{d_1,\ldots,d_t\}$?

\bibliographystyle{plain}
\bibliography{Latex}

\end{document}